\documentclass{amsart}

\usepackage{amsmath,amssymb,amsfonts,enumerate,amsthm,graphicx,color}

\newcommand{\Z}{\mbox{Z}}

\newcommand{\F}{\mathbb{F}}
\newcommand{\T}{\mathrm{T}}

\newtheorem{thm}{Theorem}[section]

\newtheorem{lem}[thm]{Lemma}

\numberwithin{equation}{section}

\begin{document}

\title{Rings whose total graphs have genus at most one}

\author{Hamid Reza Maimani}
\address{Hamid Reza Maimani\\Department of Mathematics, Shahid Rajaee University, Tehran, Iran\\
and Institute for Research in Fundamental Sciences (IPM), Tehran
Iran.}

\author{Cameron Wickham}
\address{Cameron Wickham\\Mathematics Department, Missouri State University, Springfield, MO 65897, USA.}

\author{Siamak Yassemi}
\address{Siamak Yassemi\\Department of Mathematics,
University of Tehran, Tehran, Iran \\ and Institute for Research in
Fundamental Sciences (IPM), Tehran Iran.}

\thanks{The research of Hamid Maimani was in part supported by a grant from IPM (No. 88050214)}.
\thanks{The research of Siamak Yassemi was in part supported by a grant from IPM (No. 88130213)}.



\begin{abstract}

Let $R$ be a commutative ring with $\Z(R)$ its set of zero-divisors.
In this paper, we study the total graph of $R$, denoted by
$\T(\Gamma(R))$. It is the (undirected) graph with all elements of $R$
as vertices, and for distinct $x, y\in R$, the vertices $x$ and $y$
are adjacent if and only if $x + y\in\Z(R)$. We investigate
properties of the total graph of $R$ and determine all isomorphism
classes of finite commutative rings whose total graph has genus at
most one (i.e., a planar or toroidal graph). In addition, it is shown
that, given a positive integer $g$, there are only finitely many
finite rings whose total graph has genus $g$.

\end{abstract}

\maketitle


\section*{Introduction}

Let $R$ be a commutative ring with non-zero unity. Let $\Z(R)$ be the
set of zero-divisors of $R$.  The concept of the graph of the zero
divisors of $R$ was first introduced by Beck \cite{B}, where he was
mainly interested in colorings. In his work all elements of the ring
were vertices of the graph. This investigation of colorings of a
commutative ring was then continued by D.~D.~Anderson and Naseer in
\cite{AN}. In
\cite{AL}, D.~F.~Anderson and Livingston associate a graph,
$\Gamma(R)$, to $R$ with vertices $\Z(R)\setminus\{0\}$, the set of
non-zero zero-divisors of $R$, and for distinct
$x,y\in\Z(R)\setminus\{0\}$, the vertices $x$ and $y$ are adjacent
if and only if $xy=0$.

An interesting question was proposed by Anderson, Frazier, Lauve,
and Livingston \cite{AFLL}: For which finite commutative rings $R$
is $\Gamma(R)$ planar. A partial answer was given in \cite{AMY}, but
the question remained open for local rings of order 32. In \cite{S}
and then independently in \cite{BC} and \cite{Wa} it is shown that
there is no ring of order 32 whose zero-divisor graph is planar.

The genus of a graph is the minimal integer $n$ such that the graph
can be drawn without crossing itself on a sphere with $n$ handles
(i.e. an oriented surface of genus $n$). Thus, a planar graph has
genus zero, because it can be drawn on a sphere without
self-crossing. In \cite{Wa} and \cite{Wi} the rings whose
zero-divisor graph has genus one are studied. A genus one graph is
called a toroidal graph. In other words, a graph $G$ is toroidal if it
can be embedded on the torus, that means, the graph's vertices can
be placed on a torus such that no edges cross. Usually, it is
assumed that $G$ is also non-planar. In \cite{wi2} it is shown that
for a positive integer $g$, there are only finitely many finite
rings whose zero-divisor graph has genus $g$.

In \cite{AB}, D.~F.~Anderson and Badawi introduced the total graph
of $R$, denoted by $\T(\Gamma(R))$, as the graph with all elements
of $R$ as vertices, and for distinct $x, y\in R$, the vertices $x$ and
$y$ are adjacent if and only if $x + y\in\Z(R)$.

In this paper, we investigate properties of the total graph of $R$
and determine all isomorphism classes of finite rings whose total
graph has genus at most one (i.e., a planar or toroidal graph).
In addition, we show that for a positive integer $g$, there are only
finitely many finite rings whose total graph has genus $g$.


\section{Main result}

A {\it complete graph} is a graph in which each pair of distinct
vertices is joined by an edge. We denote the complete graph with
$n$ vertices by $K_n$. A {\it bipartite graph} is a graph such that
its vertex set can be partitioned into two subsets $V_1$ and $V_2$
and each edge joins a vertex of $V_1$ to a vertex of $V_2$. A
\textit{complete bipartite graph} is a bipartite graph such that each vertex in $V_1$ is joined by an edge to each vertex in $V_2$, and is denoted by $K_{m,n}$ when $|V_1| = m$
and $|V_2| = n$. A {\it clique} of a graph is a maximal complete
subgraph. For a graph $G$, the {\it degree} of a vertex $v$ in $G$, denoted $\deg(v)$,
is the number of edges of $G$ incident with $v$. The number $\delta(G)=\min\{\deg(v)\;|\; v \mbox{ is a vertex of } G\}$ is the \textit{minimum degree} of $G$. For a nonnegative integer $k$, a graph is called {\it $k$-regular}
if every vertex has degree $k$. Recall that a graph
is said to be {\it connected} if for each pair of distinct vertices
$v$ and $w$, there is a finite sequence of distinct vertices
$v=v_1,\cdots,v_n=w$ such that each pair $\{v_i,v_{i+1}\}$ is an
edge. Such a sequence is said to be a \textit{path} and the \textit{distance}
$d(v,w)$ between connected vertices $v$ and $w$ is the length of
the shortest path connecting them.  For any graph $G$, the disjoint union of $k$ copies of $G$ is denoted $kG$.  Let $S$ be a nonempty subset of
vertex set of graph $G$. The \textit{subgraph induced by $S$} is the
subgraph with the  vertex set $S$ and with any edges whose endpoints
are both in the $S$ and is denoted by $\langle S\rangle$.

Let $G_1=(V_1,E_1)$ and $G_2=(V_2,E_2)$ be two
graphs with disjoint vertices set $V_i$ and edges set $E_i$. The
cartesian product of $G_1$ and $G_2$ is denoted by $G=G_1\times G_2$
with vertices set $V_1\times V_2$ and $(x,y)$ is adjacent to
$(x',y')$ if $x=x'$ and $y$ is adjacent $y'$ in $G_{2}$ or $y=y'$
and $x$ is adjacent to $x'$ in $G_{1 }$.

\begin{lem} Let $x$ be a vertex of $\T(\Gamma(R))$.  Then the degree of $x$ is either $|\Z(R)|$ or $|\Z(R)|-1$. In particular, if $2 \in \Z(R)$, then $\T(\Gamma(R))$ is a $(|\Z(R)|-1)$-regular graph.
\end{lem}

\begin{proof}  If $x$ is adjacent to $y$, then $x + y = a \in \Z(R)$ and hence $y = a - x$ for some $a \in \Z(R)$. We have two cases:
\\
\noindent \emph{Case 1}. Suppose $2x \in \Z(R)$.  Then $x$ is adjacent to $a - x$ for any $a \in \Z(R) \setminus \{2x\}$. Thus the degree of $x$ is $|\Z(R)|-1$. In particular, if $2 \in \Z(R)$, then $\T(\Gamma(R))$ is a $(|\Z(R)| - 1)$-regular graph.
\\
\noindent \emph{Case 2}. Suppose $2x \notin \Z(R)$.  Then $x$ is adjacent to $a-x$ for any $a \in \Z(R)$.  Thus the degree of $x$ is $|Z(R)|$. \end{proof}

Let $S_{k}$ denote the sphere with $k$ handles, where $k$ is a
non-negative integer, that is, $S_{k}$ is an oriented surface of
genus $k$.
 The genus of a graph $G$, denoted by $\gamma(G)$, is the minimum integer $n$ such that
the graph can embedded in $S_{n}$. A graph $G$ is called a planar if $\gamma(\Gamma(G))=0$, and toroidal if $\gamma(\Gamma(G))=1$.
We note here that if $H$ is a subgraph of a graph $G$, then $\gamma (H)\leq \gamma(G)$.

In the following theorem we bring some well-known formulas, see,
e.g., \cite{W} and \cite{W2}:

\begin{thm}\label{genera_bounds} The following statements hold:

\begin{itemize}

\item[(a)] For $n\geq 3$ we have $\gamma(K_{n})= \left\lceil \frac{(n-3)(n-4)}{12} \right\rceil$.

\item[(b)] For $m,n\geq 2$ we have $\gamma(K_{m,n})=\left\lceil \frac{(m-2)(n-2)}{4} \right\rceil$.

\item[(c)] Let $G_1$ and $G_2$ be two graphs and for each $i$, $p_{i}$ be the number of vertices of $G_{i}$.
Then $max \{p_{1}\gamma(G_{2})+ \gamma(G_{1}),
p_{2}\gamma(G_{1})+\gamma(G_{2})\}\leq \gamma(G_{1}\times G_{2}).$

\end{itemize}

\end{thm}

According to Theorem \ref{genera_bounds} we have  $\gamma(K_{n})=0$ for $1\leq n\leq
4$ and $\gamma(K_{n})=1$ for $5\leq n\leq 7$ and for other value of
$n$, $\gamma(K_{n})\geq 2.$

\begin{lem}\label{Z2ProductTotalGraph} Let $\F_q$ denote the field with $q$ elements.  Then the total graph of $\F_2\times \F_q$ is isomorphic to $K_2 \times K_q$.  Furthermore, for any positive integer $m$ and $q>2$,
$$
\gamma(T(\Gamma(\F_{2^m} \times \F_q))) \geq 2^m\left\lceil \frac{(q-3)(q-4)}{12}\right\rceil.
$$
\end{lem}

\begin{proof}  We induct on $m$.  If $m=1$, let $V_0 = \{ (0,y) | y \in \F_q \}$ and let $V_1 = \{ (1,y) | y \in \F_q \}$.  Then the subgraphs $G_i$ of the total graph of $\F_2 \times \F_q$ induced by each of the $V_i$ is $K_q$.  Now for each $y \in \F_q$, there is an edge between $(0, y) \in G_0$ and $(1, -y) \in G_1$.  Furthermore, these are the only other edges in the total graph.  Identifying $(1, -y)$ with $(1, y)$, we can replace $G_1$ with an isomorphic copy $G'_1$; under this isomorphism the edge between $(0, y) \in G_0$ and $(1, -y) \in G_1$ is the edge between $(0, y) \in G_0$ and $(1, y) \in G'_1$.  Thus the total graph of $\F_2 \times \F_q$ has vertex set $\{ (x, y) | x \in \F_2 \mbox{ and } y \in \F_q \}$, with an edge between $(x,y)$ and $(x',y')$ if $x = x'$  and $y \neq y'$, or $y = y'$ and $x \neq x'$.  That is, it is the graph $K_2 \times K_q$.  Parts (a) and (c) of Theorem \ref{genera_bounds} now yield
$$
\gamma(T(\Gamma(\F_2 \times \F_q))) = \gamma(K_2 \times K_q) \geq 2\left\lceil \frac{(q-3)(q-4)}{12}\right\rceil.
$$
If $m>1$, we can partition $\F_{2^m}$ into two sets, $S_1$ and $S_2$, each of cardinality $2^{m-1}$; let $f$ be a bijection from $S_1$ to $S_2$.  Since each element of a field of characteristic $2$ is its own inverse, then the subgraph of $T(\Gamma(\F_{2^m} \times \F_q))$ induced by $S_i \times \F_q$ is isomorphic to $T(\Gamma(\F_{2^{m-1}} \times \F_q))$.  For any $y \in \F_q$ and $s\in S_1$, the element $(s,y)$ is adjacent to $(f(s),-y)$.  We thus have a copy of $K_2 \times T(\Gamma(\F_{2^{m-1}} \times \F_q))$ as a subgraph of $T(\Gamma(\F_{2^m} \times \F_q))$.  Part (c) of Theorem \ref{genera_bounds} and the induction hypothesis now yield
\begin{eqnarray*}
\gamma(T(\Gamma(\F_{2^m} \times \F_q))) &\geq& \gamma(K_2 \times T(\Gamma(\F_{2^{m-1}} \times \F_q))) \\
 &\geq& 2\gamma(T(\Gamma(\F_{2^{m-1}} \times \F_q))) \\
 &\geq& 2^m\left\lceil \frac{(q-3)(q-4)}{12}\right\rceil.
\end{eqnarray*}\end{proof}

A {\it subdivision} of a graph is a graph obtained from it by replacing
edges with pairwise internally disjoint paths. A remarkably simple
characterization of planar graphs was given by Kuratowski in 1930.
Kuratowski Theorem says that a graph is planar if and only if it
contains no subdivision of $K_5$ or $K_{3,3}$ (see [2, p. 153]). In
addition, every planar graph has a vertex $v$ such that $deg(v)\leq
5$.

\begin{thm}
For any positive integer $g$, There are finitely many finite rings $R$ whose total graph has genus $g$.
\end{thm}

\begin{proof}
Let $R$ be a finite ring. If $R$ is local, then $Z(R)$ is the maximal ideal of $R$ and $|R|\leq |Z(R)|^2$.  If $R$ is not local, then $R= R_{1}\times R_{2}\times \cdots
\times R_{n}$ where each of the $R_{i}$'s is a local ring and $n\geq 2$ \cite{MC}.
Suppose that $|R_{1}|\leq |R_{2}|\leq \cdots \leq |R_{n}|$ and set
$R_{1}^{*}={0}\times R_{2}\times \cdots R_{n}$. Since $|R|=
|R_{1}||R_{1}^{*}|$, we conclude that $|R|\leq |R_{1}^{*}|^{2}$. Let $S$ denote either $Z(R)$ if $R$ is local or $R_1^*$ if $R$ is not local.  Then
every pair of elements of $S$ are adjacent in
$\T(\Gamma(R))$ and hence we have a complete graph $K_{|S|}$
in the structure of $\T(\Gamma(R))$. This implies that
$\gamma (K_{|S|}) \leq g$. Therefore $\lceil
((|S|-3)(|S|-4))/12\rceil \leq g$ and so
$|S|\leq \left(7 +\sqrt{49 + 48(g - 1)}\right)/2$, and hence $|R|\leq \left(\left(7+\sqrt{49 + 48(g - 1)}\right)/2\right)^{2}.$
\end{proof}

\begin{thm}\label{PlanarRings}
Let $R$ be a finite ring such that $\T(\Gamma (R))$ is planar. Then the
following hold:

\begin{itemize}

\item[(a)] If $R$ is local ring, then $R$ is a field or $R$
is isomorphic to the one of the $9$ following rings:
$$\mathbb{Z}_{4}, \mathbb{Z}_{2}[X]/(X^{2}), \mathbb{Z}_{2}[X]/(X^{3}),
\mathbb{Z}_{2}[X,Y]/(X,Y)^{2}, \mathbb{Z}_{4}[X]/(2X,X^{2}),$$
$$\mathbb{Z}_{4}[X]/(2X,X^{2}-2), \mathbb{Z}_{8}, \F_4[X](X^{2}), \mathbb{Z}_{4}[X]/(X^{2}+X+1).$$

\item[(b)] If $R$ is not local ring, then $R$ is an infinite integral
domain or $R$ is isomorphic to $\mathbb{Z}_2 \times \mathbb{Z}_2$ or $\mathbb{Z}_6$.

\end{itemize}

\end{thm}

\begin{proof}
Any planar graph has a vertex $v$ with $\deg(v)\leq 5$. So if the total graph of $R$ is planar, then $\delta (\T(\Gamma(R)))\leq 5$.  By Lemma 1.1, $\delta(\T(\Gamma (R)))=|Z(R)|$ or $|Z(R)|-1$, and hence $|Z(R)|\leq 6$.

{\rm (a)} Assume that $R$ is a local ring and let $n=|Z(R)|$ and $m=|R/Z(R)|$. If $2\in Z(R)$ then $\T(\Gamma(R))\cong m K_n$ (\cite[Theorems 2.1 and 2.2]{AB}).  Hence $|Z(R)| \leq 4$. Also, $|R| =2^k$ since $2\in Z(R)$. So $|R|=16, 8, 4,$ or $2$.  According to Corbas and Williams \cite{cb} there are two non-isomorphic rings of order $16$ with maximal ideals of order $4$, namely $\F_4[x]/(x^2)$ and $\mathbb{Z}_4[x]/(x^2+x+1)$ (see also Redmond \cite{R}), so for these rings we have $\T(\Gamma (R))\cong 4K_4$.  Since $K_4$ is planar we conclude that the total graphs of these
rings are planar.  In \cite{cb} it is also shown that there are $5$ local rings of order $8$ (except $F_{8}$). In all of these rings we have $|Z(R)| =4$ and hence $\T(\Gamma (R))\cong 2K_4$.  Also, there are two non-isomorphic local rings of order 4; these are $\mathbb{Z}_{4}$ and $\mathbb{Z}_{2}[X]/(X^{2})$.  For both we have $\T(\Gamma (R))\cong 2K_2$ and thus they are planar. Note that if $|Z(R)|=1$, then $R$ is a field and hence the total graph is planar. If $2\notin Z(R)$, then $\T(\Gamma(R))\cong K_n \cup \left(\frac{m-1}{2}\right) K_{n,n}$ (\cite[Theorem 2.2]{AB}).
This implies $n \leq 2$, and thus $R$ either has order $4$ or is a field.

{\rm (b)} Suppose that $R$ is not local ring. Since $R$ is finite,
then there are finite local rings $R_{i}$ such that $R=R_{1}\times
\cdots \times R_{t}$ where $t\geq 2$. Since $|Z(R)| \leq 6$
then we have the following candidates: $$\mathbb{Z}_2 \times
\mathbb{Z}_2, \mathbb{Z}_6, \mathbb{Z}_{2}\times \F_4, \mathbb{Z}_{2}\times \mathbb{Z}_{4},  \mathbb{Z}_{2}\times \mathbb{Z}_{2}[X]/(X^{2}), $$
$$\mathbb{Z}_{2}\times \mathbb{Z}_{5}, \mathbb{Z}_{3}\times
\mathbb{Z}_{3}, \mathbb{Z}_{3}\times
\F_4.$$ The total graph of $\mathbb{Z}_2
\times \mathbb{Z}_2$ is isomorphic to the cycle $C_{4}$, and this graph
is planar. By Lemma \ref{Z2ProductTotalGraph}, the total graph of $\mathbb{Z}_6 \cong \mathbb{Z}_2 \times \mathbb{Z}_3$ is isomorphic to
$K_{2}\times K_{3}$, which is also planar.

Let $R$ be a ring with $|R|=n$.  The subgraph of the total graph of $\mathbb{Z}_{2}\times R$ induced by the set $\{0\} \times R$ is a copy of $K_n$.  The edge $(1,0) - (0,0)$, together with the paths $(1,0) - (1,-r) - (0,r)$ for each $r\in R$ yield a subdivision of $K_{n+1}$ in the total graph of $\mathbb{Z}_{2}\times R$.  Thus the total graphs of $\mathbb{Z}_{2}\times \F_4$, $\mathbb{Z}_{2}\times \mathbb{Z}_4$, $\mathbb{Z}_{2}\times \Bbb{Z}_{2}[X]/(X^{2})$, and $\mathbb{Z}_{2}\times \mathbb{Z}_5$ are not planar.  Also, the total graph of $\mathbb{Z}_{3}\times R$ contains a subgraph which is isomorphic to $K_{3,n}$ (consider the induced subgraph $\langle S \rangle$ where $S=\{(1,r)\:|\: r\in R\} \cup \{(2,r)\:|\: r\in R \}$).  Thus the total graphs of $\mathbb{Z}_{3}\times \mathbb{Z}_3$ and $\mathbb{Z}_{3}\times \mathbb{F}_4$ are not planar.
\end{proof}

\begin{thm}\label{ToroidalRings}
Let $R$ be a finite ring such that $\T(\Gamma (R))$ is toroidal. Then the following statements hold:

\begin{itemize}

\item[(a)]  If R is local ring, then $R$ is isomorphic to $\Bbb{Z}_9$, or $\Bbb{Z}_{3}[X]/(X^{2})$.

\item[\rm (b)] If $R$ is not local ring, then $R$ is isomorphic to one of the following rings:
$$\Bbb{Z}_{2}\times \F_{4}, \Bbb{Z}_{3}\times \Bbb{Z}_{3}, \Bbb{Z}_{2}\times \Bbb{Z}_{4}, \Bbb{Z}_{2}\times \mathbb{Z}_{2}[x]/(x^2), \Bbb{Z}_{2}\times \Bbb{Z}_{2}\times \Bbb{Z}_{2}$$

\end{itemize}

\end{thm}

\begin{proof}
For any graph $G$ with $\nu$ vertices and genus $g$ we have $\delta(G) \leq 6+(12g-12)/\nu$. If $\gamma(G)=1$, then $\delta(G)\leq 6$ and equality holds if and only if $G$ is a triangulation of the torus and $6$-regular (see \cite[Proposition 2.1]{Wi}). If $R$ is a finite ring with toroidal total graph, then $\delta(\T(\Gamma(R)))\leq 6$, and by Lemma 1.1 $\delta(\T(\Gamma(R)))=|Z(R)|$ or $|Z(R)|-1$. Thus we conclude that $|Z(G)|\leq 7$.

(a) Let $R$ be a local ring. If $2\in Z(R)$, then $\T(\Gamma(R))$ is
a disjoint union of copies of the complete graph $K_{n}$, where
$|Z(R)|=n$. Hence $5\leq n\leq 7$. But in this case $|Z(R)|$ is a power of $2$ and thus
there are no such local rings. Now suppose that
$2\notin Z(R)$. Then $T(\Gamma(R))\cong K_n \cup \left(\frac{m-1}{2}\right)K_{n,n}$, where $n=|Z(R)|$ and $m=|R/Z(R)|$.  Thus $3\leq n \leq 4$ and since $2 \notin Z(R)$ we must have $n=3$. There are two local rings, $\mathbb{Z}_{9}$ and $\mathbb{Z}_{3}[X]/(X^{2})$, such that the
cardinality of the set of the zero-divisors is $3$;
both of these rings have total graph $K_{3}\cup K_{3,3}$ which is toroidal.

(b) Assume that $R$ is not a local ring. Since $|\Z(R)|\leq 7$, we have the following candidates for $R$ by Theorem 1.5(b):
$$
\mathbb{Z}_{2}\times \F_{4}, \mathbb{Z}_{2}\times \mathbb{Z}_{4}, \mathbb{Z}_{2}\times \mathbb{Z}_{2}[x]/(x^2), \mathbb{Z}_{2}\times \mathbb{Z}_{5}, \mathbb{Z}_{3}\times \mathbb{Z}_{3},
\mathbb{Z}_{3} \times \F_{4}, \F_{4}\times \F_{4}, \mathbb{Z}_{2}\times \mathbb{Z}_{2}\times \mathbb{Z}_{2}.
$$
By Theorem \ref{PlanarRings}, $\gamma(\T(\Gamma(\mathbb{Z}_{2}\times \F_{4}))), \gamma(\T(\Gamma(\mathbb{Z}_{2}\times \mathbb{Z}_{4}))),
\gamma(\T(\Gamma(\mathbb{Z}_{3}\times \mathbb{Z}_{3})))$ are all at least $1$.
The embeddings in Figure \ref{Embeddings} parts (a), (b), and (c) show explicitly that
$\gamma(\T(\Gamma(\mathbb{Z}_{2}\times \F_{4})))=\gamma(\T(\Gamma(\mathbb{Z}_{2}\times \mathbb{Z}_{4})))=
\gamma(\T(\Gamma(\mathbb{Z}_{3}\times \mathbb{Z}_{3})))=1$.  Since $\T(\Gamma(\mathbb{Z}_{2}\times \mathbb{Z}_{2}[x]/(x^2))) \cong \T(\Gamma(\mathbb{Z}_{2}\times \mathbb{Z}_{4}))$, then $\gamma(\T(\Gamma(\mathbb{Z}_{2}\times \mathbb{Z}_{2}[x]/(x^2))))=1$.

If we partition the elements of  $\mathbb{Z}_{2}\times \mathbb{Z}_{2}\times \mathbb{Z}_{2}$ by the four sets $V_{1}=\{(0,0,0), (1,1,1)\}$, $V_{2}=\{(1,0,0), (0,1,1)\}$, $V_{3}=\{(0,1,0) ,(1,0,1)\}$, and $V_{4}=\{(0,0,1) ,(1,1,0)\}$, it is clear that $\T(\Gamma(\mathbb{Z}_{2}\times \mathbb{Z}_{2}\times
\mathbb{Z}_{2}))= K_{2,2,2,2}$. Hence $\gamma(\mathbb{Z}_{2}\times \mathbb{Z}_{2}\times
\mathbb{Z}_{2})=1$ by \cite[Corollary 4]{DLC}

By Lemma \ref{Z2ProductTotalGraph}, $\gamma(T(\Gamma(\mathbb{Z}_{2}\times \mathbb{Z}_{5}))) \geq 2$ and thus is not toroidal.
By (the proof of) Lemma \ref{Z2ProductTotalGraph}, $\gamma(T(\Gamma(\F_{4}\times \F_{4}))) \geq 2\gamma(T(\Gamma(\mathbb{Z}_{2} \times \F_{4})))$ and so by Theorem \ref{PlanarRings} $T(\Gamma(\F_{4}\times \F_{4}))$ is not toroidal.

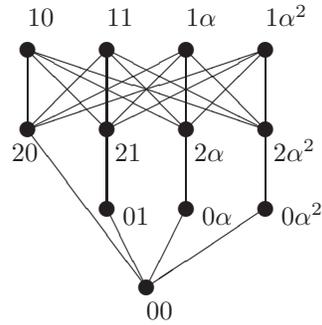
Finally, Figure \ref{Embeddings} part (d) shows a subgraph of $T(\Gamma(\mathbb{Z}_{3}\times \F_{4}))$ that is a subdivision of $K_{5,4}$, and thus by Theorem \ref{genera_bounds} $\gamma(T(\Gamma(\mathbb{Z}_{3}\times \F_{4}))) \geq 2$.
\end{proof}

\vspace{.5cm}

The authors wish to thank the referee for the detailed and useful comments.

\vfill

\pagebreak

\begin{figure}[h]
\label{Embeddings}
\caption{Embeddings in the torus and a subgraph of $T(\Gamma(\mathbb{Z}_{3}\times \F_{4}))$}
\begin{center}
\setlength{\unitlength}{3pt}
\begin{picture}(25,25)(0,0)


\put(-25,-20){\dashbox{1}(30,40)[br]}
\put(-15,20){\circle*{2}}
\put(-5,20){\circle*{2}}
\put(-16,22){$00$}
\put(-6,22){$13$}
\put(-25,10){\circle*{2}}
\put(-15,10){\circle*{2}}
\put(5,10){\circle*{2}}
\put(-30,10){$10$}
\put(-11,10){$12$}
\put(7,10){$10$}
\put(-25,0){\circle*{2}}
\put(-15,0){\circle*{2}}
\put(5,0){\circle*{2}}
\put(-30,0){$11$}
\put(-13,0){$02$}
\put(7,0){$11$}
\put(-25,-10){\circle*{2}}
\put(-5,-10){\circle*{2}}
\put(5,-10){\circle*{2}}
\put(-30,-10){$01$}
\put(-11,-10){$03$}
\put(7,-10){$01$}
\put(-15,-20){\circle*{2}}
\put(-5,-20){\circle*{2}}
\put(-16,-24){$00$}
\put(-6,-24){$13$}
\put(-15,20){\line(0,-1){10}}
\put(-25,10){\line(1,1){10}}
\put(-25,10){\line(0,-1){20}}
\put(-25,10){\line(1,0){10}}
\put(-15,10){\line(1,1){10}}
\put(-15,10){\line(1,1){10}}
\put(-15,10){\line(2,-1){20}}
\put(5,10){\line(-1,1){10}}
\put(5,10){\line(0,-1){20}}
\put(-25,10){\line(1,-1){20}}
\put(-15,0){\line(0,1){10}}
\put(-15,0){\line(-1,-1){10}}
\put(5,0){\line(-1,2){10}}
\put(5,0){\line(0,1){10}}
\put(-25,-10){\line(1,-1){10}}
\put(-15,-20){\line(0,1){20}}
\put(-15,-20){\line(1,1){20}}
\put(-5,-20){\line(0,1){10}}
\put(-5,-10){\line(1,1){10}}
\put(-5,-10){\line(-1,1){10}}
\put(-5,-10){\line(1,0){10}}
\put(-5,-20){\line(1,1){10}}
\put(-30,-30){(a) Embedding of $\mathbb{Z}_{2}\times \mathbb{Z}_{4}$}


\put(25,-20){\dashbox{1}(30,40)[br]}
\put(35,20){\circle*{2}}
\put(45,20){\circle*{2}}
\put(34,22){$00$}
\put(44,22){$1\alpha^2$}
\put(25,10){\circle*{2}}
\put(35,10){\circle*{2}}
\put(55,10){\circle*{2}}
\put(20,10){$10$}
\put(38,10){$1\alpha$}
\put(58,10){$10$}
\put(25,0){\circle*{2}}
\put(35,0){\circle*{2}}
\put(55,0){\circle*{2}}
\put(20,0){$11$}
\put(38,0){$0\alpha$}
\put(58,0){$11$}
\put(25,-10){\circle*{2}}
\put(45,-10){\circle*{2}}
\put(55,-10){\circle*{2}}
\put(20,-10){$01$}
\put(40,-10){$0\alpha^2$}
\put(58,-10){$01$}
\put(35,-20){\circle*{2}}
\put(45,-20){\circle*{2}}
\put(34,-25){$00$}
\put(44,-25){$1\alpha^2$}
\put(25,10){\line(1,1){10}}
\put(25,10){\line(1,0){10}}
\put(25,10){\line(0,-1){10}}
\put(35,10){\line(1,1){10}}
\put(55,10){\line(-1,1){10}}
\put(25,0){\line(1,1){10}}

\put(25,0){\line(0,-1){10}}

\put(35,0){\line(0,1){10}}
\put(55,0){\line(-1,2){10}}
\put(55,0){\line(0,1){10}}
\put(55,0){\line(0,-1){10}}
\put(35,0){\line(-1,-1){10}}
\put(35,0){\line(1,-1){10}}
\put(25,-10){\line(1,-1){10}}
\put(45,-10){\line(1,0){10}}
\put(55,-10){\line(-2,1){20}}
\put(35,-20){\line(0,1){20}}
\put(35,-20){\line(1,1){10}}
\put(45,-10){\line(0,-1){10}}
\put(20,-30){(b) Embedding of $ \mathbb{Z}_{2}\times \F_{4}$}
\end{picture}

\vspace{4cm}


\begin{picture}(25,25)(0,0)
\put(-25,-20){\dashbox{1}(30,40)[br]}
\put(-15,20){\circle*{2}}\put(-15,23){$11$}
\put(-5,20){\circle*{2}}\put(-5,23){$12$}
\put(-25,10){\circle*{2}}\put(-30,10){$22$}
\put(5,10){\circle*{2}}\put(8,10){$22$}
\put(-25,0){\circle*{2}}\put(-30,0){$21$}
\put(-15,0){\circle*{2}}\put(-19,-3){$01$}
\put(-5,0){\circle*{2}}\put(-2,-3){$10$}
\put(5,0){\circle*{2}}\put(8,0){$21$}
\put(-10,-5){\circle*{2}}\put(-20,-10){$02$}
\put(-15,-10){\circle*{2}}\put(-12,-10){$00$}
\put(-5,-10){\circle*{2}}\put(-2,-10){$20$}
\put(-15,-20){\circle*{2}}\put(-16,-25){$11$}
\put(-5,-20){\circle*{2}}\put(-6,-25){$12$}
\put(-15,20){\line(-1,-1){10}}
\put(-15,20){\line(-1,-2){10}}
\put(-15,20){\line(1,0){10}}
\put(-5,20){\line(-1,-1){20}}
\put(-5,20){\line(1,-1){10}}
\put(-5,20){\line(-1,-2){10}}
\put(-25,10){\line(0,-1){10}}
\put(5,10){\line(-2,-1){20}}
\put(5,10){\line(-1,-1){10}}
\put(5,10){\line(0,-1){10}}
\put(-15,0){\line(1,-1){10}}
\put(-5,0){\line(-1,-1){10}}
\put(-15,0){\line(0,-1){20}}
\put(-5,0){\line(0,-1){20}}
\put(-25,0){\line(1,-1){10}}
\put(-5,-10){\line(-1,-1){10}}
\put(-5,0){\line(1,0){10}}
\put(-15,-20){\line(1,0){10}}
\put(-30,-30){(c) Embedding of $\mathbb{Z}_{3}\times \mathbb{Z}_{3}$}


\put(20,20){\circle*{2}}\put(20,23){$10$}
\put(30,20){\circle*{2}}\put(30,23){$11$}
\put(40,20){\circle*{2}}\put(40,23){$1\alpha$}
\put(50,20){\circle*{2}}\put(50,23){$1\alpha^2$}
\put(20,10){\circle*{2}}\put(18,6){$20$}
\put(30,10){\circle*{2}}\put(31,6){$21$}
\put(40,10){\circle*{2}}\put(41,6){$2\alpha$}
\put(50,10){\circle*{2}}\put(51,6){$2\alpha^2$}
\put(30,0){\circle*{2}}\put(32,-2){$01$}
\put(40,0){\circle*{2}}\put(42,-2){$0\alpha$}
\put(50,0){\circle*{2}}\put(52,-2){$0\alpha^2$}
\put(35,-10){\circle*{2}}\put(35,-14){$00$}
\put(20,20){\line(0,-1){10}}
\put(20,20){\line(1,-1){10}}
\put(20,20){\line(2,-1){20}}
\put(20,20){\line(3,-1){30}}
\put(30,20){\line(-1,-1){10}}
\put(30,20){\line(0,-1){10}}
\put(30,20){\line(1,-1){10}}
\put(30,20){\line(2,-1){20}}
\put(40,20){\line(-2,-1){20}}
\put(40,20){\line(-1,-1){10}}
\put(40,20){\line(0,-1){10}}
\put(40,20){\line(1,-1){10}}
\put(50,20){\line(-3,-1){30}}
\put(50,20){\line(-2,-1){20}}
\put(50,20){\line(-1,-1){10}}
\put(50,20){\line(0,-1){10}}
\put(20,10){\line(3,-4){15}}
\put(30,10){\line(0,-1){10}}
\put(40,10){\line(0,-1){10}}
\put(50,10){\line(0,-1){10}}
\put(30,0){\line(1,-2){5}}
\put(40,0){\line(-1,-2){5}}
\put(50,0){\line(-3,-2){15}}
\put(20,-30){(d) A subgraph of $T(\Gamma(\mathbb{Z}_{3}\times \F_{4}))$}
\put(25,-34){that is a subdivision of $K_{5,4}$}
\end{picture}
\end{center}
\end{figure}

\pagebreak


\end{document}